\documentclass{article}
\usepackage[utf8]{inputenc}
\usepackage{amsmath}
\usepackage{amsfonts}
\usepackage{amssymb}
\usepackage{amsthm}
\usepackage[all]{xy}
\usepackage{enumerate}
\usepackage{cite}
\usepackage[affil-it]{authblk}
\usepackage{hyperref}
\hypersetup{hypertex=true,
            linkcolor=blue,
            anchorcolor=blue,
            citecolor=blue}

\theoremstyle{plain}
  \newtheorem{thm}{Theorem}[section]
  \newtheorem{lem}[thm]{Lemma}
  \newtheorem{prop}[thm]{Proposition}
  \newtheorem{cor}[thm]{Corollary} 
\theoremstyle{definition}
  \newtheorem{defn}[thm]{Definition}
  \newtheorem{rmk}[thm]{Remark}
  \newtheorem{ex}[thm]{Example}

\theoremstyle{plain}

\DeclareMathOperator{\im}{im}
  
\numberwithin{equation}{section}

\allowdisplaybreaks

\title{Formality of Sphere Bundles}
\author{Jiawei Zhou}
\date{April 18, 2024}

\begin{document}

\maketitle

\begin{abstract}
We study the formality of orientable sphere bundles over connected compact manifolds. When the base manifold is formal, we prove that the formality of the bundle is equivalent to the vanishing of the Bianchi-Massey tensor introduced by Crowley-Nordstr\"{o}m. As an example, this implies that the unit tangent bundle over a formal manifold can only be formal when the base manifold has vanishing Euler characteristic or a rational cohomology ring generated by one element. When the base manifold is not formal, we give an obstruction to the formality of sphere bundles whose Euler class is reducible.
\end{abstract}



%\section*{Statements and Declarations}

%\textbf{Conflict of interest.} The author states that there is no conflict of interest.

\section{Introduction}

In rational homotopy theory, the equivalence relation between simply connected topological spaces is defined by a continuous map $f:X\to Y$ inducing an isomorphism $\pi_*(X)\otimes\mathbb{Q}\to\pi_*(Y)\otimes\mathbb{Q}$ between the homotopy groups of their rationalizations. This condition on $f$ is equivalent to $f^*:H^*(Y,\mathbb{Q})\to H^*(X,\mathbb{Q})$ being an isomorphism. Based on the latter, Sullivan \cite{Sullivan} used a commutative differential graded algebra (CDGA) as a model of these spaces. Such CDGAs are composed of polynomial differential forms, or just differential forms if the space is a smooth manifold. The equivalence between CDGAs are also defined by quasi-isomorphisms, whose induced maps on cohomologies are isomorphisms. This equivalence relation on CDGA models also turn out to provide a classification of non-simply connected spaces.

A CDGA is called formal if it is equivalent to its cohomology, and a topological space is called formal if its CDGA model is formal. In short, the rational homotopy type of a formal space can be described just by its cohomology ring, a much simpler structure.

Several examples of formal spaces stand out, for example, H-spaces, symmetric spaces, products of formal spaces, and $k$-connected compact manifolds whose total dimension does not exceed $4k+2$ \cite{miller}. Additionally, compact K\"ahler manifolds are formal, as was first shown by Deligne, Griffiths, Morgan and Sullivan \cite{deli}. In general, a compact complex manifold is formal if it satisfies the $dd^c$-lemma. On the other hand, the $dd^{\Lambda}$-lemma, an analogous condition for symplectic manifolds which is equivalent to satisfy the hard Lefschetz property, does not imply formality, even if the manifold is simply connected \cite{caval}.

For general spaces, Sullivan in \cite{Sullivan} (see also \cite{deli}) used a special class of algebras to classify them. It consists of free graded algebras, each with a well-ordered set of generators. The differential of each generator belongs to the subalgebra generated by elements with smaller indices. Such algebras are referred to as the Sullivan algebras. Every connected CDGA is equivalent to a Sullivan algebra uniquely up to quasi-isomorphism, and a special type of Sullivan algebra uniquely up to isomorphism, which is called the minimal Sullivan algebra. If the CDGA is a model of some space, its equivalent Sullivan algebra and minimal Sullivan algebra are called the Sullivan model and the minimal Sullivan model of this space respectively.

Adding a generator of odd degree $k$ to the Sullivan model of a manifold is equivalent to constructing an orientable $S^k$-bundle over it. So we are interested in investigating how this process affects formality, and potentially applying this knowledge to study the formality for general fiber bundles. In this paper, manifolds are generally assumed to be smooth, connected and compact. Sphere bundles are assumed to be orientable, and moreover, the cohomologies are assumed to be over a field $\mathbb{K}$ of characteristic 0. Most of our proofs should also be applicable for general topological manifolds and spherical Serre fibrations since they do not rely on the smooth structure or the bundle structure.

The relationship between the formalities of the base and the total space was studied by Lupton \cite{Lupton} and Amann-Kapovitch \cite{AK}. For a fibration of simply connected spaces of finite type, if the fiber $F$ is formal, rationally elliptic (i.e. both $\pi_*(F)\otimes\mathbb{Q}$ and $H^*(F,\mathbb{Q})$ are finite dimensional), and also satisfies the Halperin conjecture (i.e. that for any fibration $F\to E\to B$ with the base $B$ being simply connected, the induced morphism $H^*(E,\mathbb{Q})\to H^*(F,\mathbb{Q})$ is surjective), then the base is formal if and only if the total space is formal. This condition holds for bundles with even-dimensional sphere fibers.

Both Lupton's and Amann-Kapovitch's proof rely on the fact that when the fiber satisfies the Halperin conjecture, the cohomology ring of the total space can be expressed as a tensor product of the cohomologies of the base and that of the fiber. Since the fiber being an odd dimensional sphere does not satisfy this condition, the cohomology of the total space becomes more intricate. Actually, there are simple examples of a non-formal bundle over a formal manifold with the fiber being an odd-dimensional sphere. Such is the case of an orientable circle bundle over a torus with a non-trivial Euler class.

However, adding one additional generator from the fiber does not break formality dramatically. In the case of a Boothby-Wang fibration, which is a circle bundle over a symplectic manifold with Euler class being the symplectic structure, Biswas, Fern\'andes, Mu\~noz and Tralle \cite[Proposition 4.5]{biswas-et-al} proved that the higher order (greater than 3) Massey products of the total space vanish, if the base symplectic manifold is formal and satisfying the hard Lefschetz property. In an earlier work \cite{own}, we also showed that a sphere bundle over a formal manifold has an $A_{\infty}$-minimal model where the only non-trivial operations are $m_2$ and $m_3$.

Regarding $A_{\infty}$-algebra, it is said to be formal if it has an $A_{\infty}$-minimal model with only $m_2$ non-trivial. The information of $m_3$ is encoded in the Bianchi-Massey tensor introduced by Crowley and Nordstr\"{o}m \cite{cn}, which is a linear map from a subspace of $(H^*)^{\otimes 4}$ to $H^*$. More precisely, a compact manifold whose Bianchi-Massey tensor vanishes has an $A_{\infty}$-minimal model with $m_3=0$. Therefore, it is natural to conjecture that a sphere bundle over a formal manifold is formal if its Bianchi-Massey tensor vanishes.

Unlike the possible different representatives of $A_\infty$-minimal models and Massey products, the Bianchi-Massey tensor is uniquely defined without ambiguity. So once we prove the above conjecture in Section 3, it follows that the formality of a sphere bundle over a formal manifold can be determined by finite calculation. It turns out that we can construct a CDGA equivalence between the differential forms of the sphere bundle and its cohomology, and our proof will not use the language of $A_\infty$-algebra.

\begin{thm}\label{intro formality over formal}
Suppose $M$ is a compact formal manifold, and $\pi:X\to M$ is an orientable $S^k$-bundle. Then $X$ is formal if and only if the Bianchi-Massey tensor of $\Omega^*(X)$ vanishes. Moreover, when $k$ is even, $X$ is always formal.
\end{thm}

It turns out that a trivial Euler class is a sufficient condition for the formality of a sphere bundle over a formal manifold. For non-trivial Euler classes, we can consider the special case that the Euler class is of top degree. This requires the manifold to be even-dimensional. We prove that such bundles are formal only when the rational cohomology rings of the base manifolds are generated by one element. This is a consequence of Theorem $\ref{formality of extended by top degree}$.

\begin{thm}\label{intro unit tangent bundle}
Suppose that $M$ is a $2k$-dimensional compact orientable formal manifold. Let $X$ be an $S^{2k-1}$-bundle with non-trivial Euler class. If $X$ is formal, then $H^*(M)=\mathbb{K}[x]/(x^p)$ is a quotient of the polynomial ring with a single variable.
\end{thm}

It immediately follows that if the unit tangent bundle of a formal manifold $M$ is formal, then either the Euler characteristic $\chi(M)=0$, or $H^*(M)$ is generated by one element.

It is also interesting to consider the more general case when the base $M$ is not necessarily formal. As $\Omega^*(M)$ may no longer be represented by the cohomology, we shall consider its minimal Sullivan model, labelled by $\mathcal{M}$, instead. Depending on the reducibility of the representative of the Euler class in $\mathcal{M}$, the minimal Sullivan model of the sphere bundles can have two types. In this paper, we mainly consider the reducible case, and give an obstruction to formality under certain conditions.

\begin{thm}\label{intro HL obstruction}
Let $(M,\omega)$ be a connected symplectic manifold satisfying the hard Lefschetz property. Suppose $[\omega]$ is an integral and reducible cohomology class, i.e. there exists some $x_i,y_i\in H^1(M)$ such that $[\omega]=\sum [x_i]\wedge[y_i]$. Then the Boothby-Wang fibration of $M$, i.e. the circle bundle with Euler class $[\omega]$, is non-formal.
\end{thm}

This obstruction can be generalized to $S^{4k+1}$-bundles, as long as the Euler class $[\omega]$ can be written as a sum of products of $(2k+1)$-cohomology classes, and the hard Lefschetz property can be weakened in the following way: There exists some $s\geq 0$ such that taking the product with $[\omega]$ is an isomorphism from a non-trivial space $H^s(M)$ to $H^{s+4k+2}(M)$, and also, is injective from $H^{s-2k-1}(M)$ to $H^{s+2k+1}(M)$. The condition of the sphere dimension being $4k+1$ is difficult to relax. We will give a simple example of a formal $S^{4k+3}$-bundle satisfying other requirements.  

This paper is organized as follows. In Section 2, we review the algebraic tools that play a central role in this paper, including the Bianchi-Massey tensor and the minimal Sullivan model.  We give the proof of Theorem \ref{intro formality over formal} and  Theorem \ref{intro unit tangent bundle} in Section 3. And in Section 4, we discuss the formality of general sphere bundles, and give the proof of Theorem \ref{intro HL obstruction}. 

\textbf{Acknowledgement.} The author thanks Ruizhi Huang, Si Li, Jianfeng Lin, Li-Sheng Tseng and Jie Wu for helpful discussions and valuable suggestions. The author is also grateful to the referee for carefully reading the paper and giving many helpful suggestions for improvement, especially for pointing out and resolving an issue with the proof of Lemma \ref{vanishing uniform massey}, and suggesting a generalization of Lemma \ref{vanishing uniform massey to formal}. The author would like to acknowledge the support of the National Key Research and Development Program of China No.~2020YFA0713000.

\section{Preliminary}

\subsection{Bianchi-Massey tensor}
Let $V$ be a graded vector space over a field $\mathbb{K}$ of characteristic 0. We let $\mathcal{G}^kV$ denote the $k$-th graded symmetric power of $V$, i.e. the quotient space of $V^{\otimes k}$ by the relations of graded commutativity. We will use $(x_1\cdot x_2)$ denote the graded symmetric product of $x_1$ and $x_2$. Hence, $(x_1\cdot x_2)=(-1)^{|x_1||x_2|}(x_2\cdot x_1)$, where $|x_1|,|x_2|$ are the degree of $x_1$ and $x_2$ respectively. We can define $(x_1\cdot x_2\cdot \ldots\cdot x_k)$ similarly, and such elements generate $\mathcal{G}^kV$ when all $x_i\in V$.

\begin{rmk}
In our setting, $\mathcal{G}^kV$ is isomorphic to the space of graded commutative $k$-tensors of $V$, although this is not true if the characteristic of the ground field is nonzero or $V$ is replaced by an abelian group.
\end{rmk}

Let $K[\bullet]$ denote the kernel of a tensor space under full graded symmetrisation. For example, if $V$ is a graded vector space, $K[\mathcal{G}^2\mathcal{G}^2V]$ is the kernel of the following symmetrisation
$$
\mathcal{G}^2\mathcal{G}^2V\to \mathcal{G}^4V,\quad \big( (x\cdot y) \cdot (z\cdot w) \big) \mapsto (x\cdot y\cdot z\cdot w).
$$
Thus, $(x\cdot y)(z\cdot w)-(-1)^{|y||z|}(x\cdot z)(y\cdot w)\in K[\mathcal{G}^2\mathcal{G}^2V]$.

Now suppose $\mathcal{A}$ is a CDGA (commutative differential graded algebra). Let
$$
c:\mathcal{G}^2H^*(\mathcal{A})\to H^*(\mathcal{A}), \quad (x\cdot y)\mapsto xy
$$
denote the product map, and $E^*(\mathcal{A})=\ker c$, which we will simply write as $E^*$. Then we set
$$
\mathcal{B}^*(H^*(\mathcal{A}))=\mathcal{G}^2E^*\cap K[\mathcal{G}^2\mathcal{G}^2H^*(\mathcal{A})].
$$
For simplicity, we will use $\mathcal{B}^*(\mathcal{A})$ to denote $\mathcal{B}^*(H^*(\mathcal{A}))$.

Let $\mathcal{Z}^*(\mathcal{A})\subset \mathcal{A}^*$ be the subspace of $d$-closed elements. Pick a right inverse $\alpha:H^*(\mathcal{A})\to \mathcal{Z}^*(\mathcal{A})$ for the projection to cohomology. Then the map
$$
\alpha^2:\mathcal{G}^2H^*(\mathcal{A})\to \mathcal{A}^*, \quad (x\cdot y)\mapsto \alpha(x)\alpha(y)
$$
takes exact values on $E^*$. So there exists a linear map $\gamma:E^*\to \mathcal{A}^{*-1}$ satisfying $d\gamma=\alpha^2$.

\begin{defn}
The map
$$
\mathcal{G}^2E^*\to \mathcal{A}^{*-1},\quad (e\cdot e')\mapsto \gamma(e)\alpha^2(e')+(-1)^{|e||e'|}\gamma(e')\alpha^2(e)
$$
takes closed values on $\mathcal{B}^*(A)$. So it induces a map
$$
\mathcal{F}:\mathcal{B}^*(\mathcal{A})\to H^{*-1}(\mathcal{A}),
$$
which is called the \textbf{Bianchi-Massey tensor}. This map is independent of the choices of $\alpha$ and $\gamma$.
\end{defn}

Another obstruction to formality is the uniform Massey triple product. It is defined as follows.
\begin{defn}\label{def of uniform Massey}
Given choices of $\alpha$ and $\gamma$ as before, the map
$$
\gamma\alpha:E^*\otimes H^*(\mathcal{A})\to \mathcal{A}^{*-1}, \quad e\otimes x\mapsto \gamma(e)\alpha(x)
$$
takes closed values on $K[E^*\otimes H^*(\mathcal{A})]$, which is the kernel of the full graded symmetrisation $E^*\otimes H^*(\mathcal{A})\to\mathcal{G}^3H^*(\mathcal{A})$. So it induces a map
$$
\mathcal{T}:K[E^*\otimes H^*(\mathcal{A})]\to H^{*-1}(\mathcal{A}),
$$
which is called the \textbf{uniform Massey triple product}.
\end{defn}

Unlike the Bianchi-Massey tensor, the uniform Massey triple product does depend on the choices of $\alpha$ and $\gamma$.

\subsection{Minimal Sullivan model}

\begin{defn}
A \textbf{Sullivan algebra} is a CDGA $\mathcal{M}=\Lambda V^*$ which is free as a graded algebra. $V^*$ has a homogeneous basis $\{v_{\alpha}\}$ indexed by a well-ordered set, such that the degree $|v_{\alpha}|\geq 1$ and $dv_{\alpha}\in\Lambda V^*_{<\alpha}$. Here $V^*_{<\alpha}$ is spanned by $v_{\beta}$ with $\beta<\alpha$.

If a Sullivan algebra also satisfies $\beta<\alpha$ whenever $|v_{\beta}|<|v_{\alpha}|$, it is called \textbf{minimal}.
\end{defn}

In a minimal Sullivan algebra $\mathcal{M}=\Lambda V^*$, all $dv_{\alpha}$ are \textbf{reducible}, i.e. there exist $x_i,y_i\in\mathcal{M}^+$ such that $dv_{\alpha}=\sum x_iy_i$. Here $\mathcal{M}^+$ is the subspace of $\mathcal{M}$ spanned by elements with degree at least 1. Equivalently, we can say that all $dv_{\alpha}$ are in $\Lambda^{\geq 2} V$, the subspace generated by elements of wordlength at least 2.

\begin{defn}
If a CDGA $\mathcal{A}$ is quasi-isomorphic to a (minimal) Sullivan algebra $\mathcal{M}$, we call $\mathcal{M}$ a \textbf{(minimal) Sullivan model} of $\mathcal{A}$. For a manifold $M$, say $\mathcal{M}$ is a (minimal) Sullivan model of $M$ if it is a (minimal) Sullivan model of $\Omega^*(M)$.
\end{defn}

Every simply-connected manifold has a minimal Sullivan model generated by a graded vector space $V^*$ of finite type, i.e. the subspace $V^k$ of degree $k$ is finite dimensional for all $k$. Moreover, the degrees of all elements in $V^*$ are at least 2 \cite{deli}. More generally, every connected CDGA has a minimal Sullivan model generated by $V^*$, and the degrees of all elements in $V^*$ are at least 1 \cite{Halperin}. The minimal Sullivan model is unique up to isomorphism.

\begin{thm}[Deligne-Griffiths-Morgan-Sullivan \cite{deli}]\label{C oplus N decomposition}
A minimal Sullivan algebra $\mathcal{M}$ is formal if and only if the following statement holds. $\mathcal{M}$ is generated by $V^*=C^*\oplus N^*$, where $C^*$ is the subspace of closed elements in $V^*$ and $N^*$ is a direct complement of $C^*$, and all the closed elements in the ideal $\mathbf{I}(N^*)$ generated by $N^*$ are exact.
\end{thm}

\section{Formality of Sphere Bundles over Formal Spaces}

\subsection{The Bianchi-Massey tensor determines formality}
Let $\omega\in \mathcal{A}^{k}$ be a closed element in a CDGA $\mathcal{A}$. Suppose $\theta$ has degree $k-1$, and satisfies $\theta^2=\theta\theta=0$ (This trivially holds when $k-1$ is odd). We can extend $\mathcal{A}$ by $\theta$:
$$
\mathcal{A}_{\omega}=\mathcal{A}\otimes \Lambda\theta=\{ x+\theta y \mid x,y\in \mathcal{A} \}
$$
with $d\theta=\omega$. Here $\Lambda\theta=\langle 1,\theta \rangle$ is an exterior algebra generated by $\theta$.

$\mathcal{A}_{\omega}$ is also a CDGA. As the following lemma shows, its isomorphism class depends only on the cohomology class $[\omega]\in H^*(\mathcal{A})$ in numerous instances.

\begin{lem}\label{choosing representatives of extension}
Suppose $\mathcal{A}$ is a connected CDGA, i.e. $H^0(\mathcal{A})=\mathbb{K}$. Take an even-dimensional, closed but non-exact element $\omega\in\mathcal{A}$, then construct $\mathcal{A}_{\omega}$ as above. Let $\bar{\omega}\in\mathcal{A}$ be an arbitrary representative of $[\omega]\in H^*(\mathcal{A})$, and $\bar{\theta}$ be any element in $\mathcal{A}_{\omega}$ satisfying $d\bar{\theta}=\bar{\omega}$. Then $\mathcal{A}_{\omega}$ can also be written as $\mathcal{A}_{\bar{\omega}}=\mathcal{A}\otimes \Lambda\bar{\theta}$.
\end{lem}
\begin{proof}
Write $\bar{\theta}=\xi+\theta\eta$ for some $\xi,\eta\in\mathcal{A}$, then
$$
\bar{\omega}-\omega-d\xi=d(\theta(\eta-1))=(\eta-1)\omega-\theta d\eta
$$
is exact in $\mathcal{A}$. Thus, $d\eta=0$ and $[(\eta-1)\omega]=0$ in $H^*(\mathcal{A})$. Since $H^0(\mathcal{A})=\mathbb{K}$, the closed element $\eta\in\mathcal{A}^0$ has to be a constant. Moreover, $[\omega]\in H^*(\mathcal{A})$ is a non-trivial class by hypothesis, so $\eta$ can only be 1. Hence, $\bar{\theta}$ is in the affine space $\theta+\mathcal{A}$. Then we can write $\mathcal{A}_{\omega}=\mathcal{A}\oplus\bar{\theta}\mathcal{A}$ as a graded vector space.

By hypothesis $\omega$ is of even degree, so the degree of $\bar{\theta}$ is odd and $\bar{\theta}^2=0$. This fact together with the algebraic structure on $\mathcal{A}_{\omega}$ identify this CDGA and $\mathcal{A}\otimes \Lambda\bar{\theta}$.
\end{proof}

When $\mathcal{A}$ is formal, there exists a zigzag of quasi-isomorphisms connecting $\mathcal{A}$ and $H^*(\mathcal{A})$. These quasi-isomorphisms are naturally extended to quasi-isomorphisms connecting $\mathcal{A}_{\omega}$ and $H^*(\mathcal{A})\otimes \Lambda\theta$. For the later CDGA we set $d\theta=[\omega]$. Thus, the formality of $\mathcal{A}_{\omega}$ is determined by $H^*(\mathcal{A})\otimes \Lambda\theta$ as long as $\mathcal{A}$ is formal.

Therefore, it is sufficient to consider the CDGA $\mathcal{A}$ with trivial differential $d=0$. In this case, the space of exact elements in $\mathcal{A}_{\omega}$ is $\im\omega$, the image of the map by left multiplying $\omega$
$$
\omega:\mathcal{A}^i\to\mathcal{A}^{i+k},\quad x\mapsto\omega x.
$$
The space of closed elements in $\mathcal{A}_{\omega}$ is $\mathcal{A}\oplus \theta\ker\omega$. Thus,
$$
H^*(\mathcal{A}_{\omega})\simeq \text{coker}\,\omega\oplus\theta\ker\omega.
$$

\begin{defn} 
Let $H^*$ be a finite dimensional graded commutative algebra over $\mathbb{K}$. If there exists some $\alpha_H\in (H^n)^{\vee}$, the dual space of $H^n$, such that the linear map
$$
\alpha_H\frown:H^i\to (H^{m-i})^{\vee},\quad x\mapsto\big(y\mapsto\alpha_H(xy)\big)
$$
is an isomorphism for all $i$, then we say that $H$ is \textbf{$n$-dimensional Poincar\'e} and $\alpha_H$ is a \textbf{Poincar\'e class}.

A CDGA is called \textbf{$n$-dimensional Poincar\'e} if its cohomology is.
\end{defn}

For a Poincar\'e CDGA, the Bianchi-Massey tensor $\mathcal{F}$ and the uniform Massey product $\mathcal{T}$ are equivalent \cite[Lemma 2.8]{cn}. In particular, if $\mathcal{F}$ vanishes, we can choose $\alpha$ and $\gamma$ such that $\mathcal{T}=0$. Based on their proof we will show that, when $\mathcal{A}_{\omega}$ is a Poincar\'e CDGA, there exists such a choice also satisfying $\im\gamma\subset\mathbf{I}(\theta)$. Here $\mathbf{I}(\theta)$ is the ideal generated by $\theta$.

\begin{lem}\label{vanishing uniform massey}
Let $\mathcal{A}$ be a CDGA with trivial differential. Suppose $\mathcal{A}_{\omega}$ is $n$-dimensional Poincar\'e and the Bianchi Massey tensor $\mathcal{F}:\mathcal{B}^{n+1}(\mathcal{A}_{\omega})\to H^n(\mathcal{A}_{\omega})$ vanishes. There are choices of $\alpha$ and $\gamma$ such that $\im\gamma\subset \mathbf{I}(\theta)$ and the morphism
$$
\gamma\alpha: E^*(\mathcal{A}_{\omega})\otimes H^*(\mathcal{A}_{\omega})\to \mathcal{A}_{\omega}^{*-1},\quad e\otimes a\mapsto \gamma(e)\alpha(a)
$$
is trivial on $K[E^*(\mathcal{A}_{\omega})\otimes H^*(\mathcal{A}_{\omega})]$.
\end{lem}

\begin{proof}
We will use $E^*$ to denote $E^*(\mathcal{A}_{\omega})$ in this proof. It is sufficient to construct $\alpha$ and $\gamma$ such that $\im\gamma\subset \mathbf{I}(\theta)$ and the uniform Massey product $\mathcal{T}=0$ on $K[E^*\otimes H^*(\mathcal{A}_{\omega})]$. Then the image of $\gamma\alpha$ are exact elements in $\mathbf{I}(\theta)$. But as discussed earlier in this section, the exact elements in $\mathcal{A}_{\omega}$ are in $\im\omega$. So $\gamma\alpha$ has to be 0 on $K[E^*\otimes H^*(\mathcal{A}_{\omega})]$.

Pick a right inverse of the left multiplication map $\omega:\mathcal{A}\to\im\omega$ and extend it to an endomorphism on $\mathcal{A}$. We use $\omega^{-1}$ to denote this map.

Choose an $\alpha:H^*(\mathcal{A}_{\omega})\to \mathcal{A}_{\omega}^*$ such that $\alpha(\theta\ker\omega)\subset\mathbf{I}(\theta)$. As $\alpha^2(E^*)$ consists of exact forms, it is in $\im\omega\subset\mathcal{A}$. So we can set
$$
\gamma=\theta\circ\omega^{-1}\circ\alpha^2:E^*\to \mathcal{A}_{\omega}^{*-1}.
$$
Then $\im\gamma\subset \mathbf{I}(\theta)$. It follows that $\im\mathcal{T}\subset\theta\ker\omega$. For an extreme case that $\ker\omega=0$, we already have $\mathcal{T}=0$. Note that in this case $\mathcal{A}$ is infinite dimensional, unless $\mathcal{A}=0$.

From now on we consider the general case that $\theta\ker\omega$ is non-trivial. Then $\theta\ker\omega$ must contain $H^n(\mathcal{A}_{\omega})$, because for non-trivial $\theta x\in\theta\ker\omega$ there exists $y\in H^*(\mathcal{A}_{\omega})$ such that
$$
\alpha_H(\theta xy)=(\alpha_H\frown \theta x)y\neq 0.
$$

To make $\mathcal{T}=0$, we will change $\gamma$ to $\gamma'=\gamma+\eta$ for some $\eta:E^*\to \mathcal{Z}^{*-1}(\mathcal{A}_{\omega})$ satisfying $\im\eta\subset \mathbf{I}(\theta)$.

Consider the map
$$
\gamma\alpha^2:E^*\otimes H^*(\mathcal{A}_{\omega}) \otimes H^*(\mathcal{A}_{\omega})\to \mathcal{A}_{\omega}^{*-1},\quad e\otimes x\otimes y\mapsto \gamma(e)\alpha(x)\alpha(y).
$$
It takes closed values on $K[E^*\otimes H^*(\mathcal{A}_{\omega}) \otimes H^*(\mathcal{A}_{\omega})]$, and factors through the projection $E^*\otimes H^*(\mathcal{A}_{\omega}) \otimes H^*(\mathcal{A}_{\omega})\to E^*\otimes \mathcal{G}^2H^*(\mathcal{A}_{\omega})$. Hence, $\gamma\alpha^2$ induces another map
$$
\mu:K[E^*\otimes \mathcal{G}^2H^*(\mathcal{A}_{\omega})]\to H^{*-1}(\mathcal{A}_{\omega}).
$$
$\mu$ depends on the choice of $\gamma$. Our goal is to find some $\gamma'$ such that the corresponding $\mu'=0$ when acting on the degree $n+1$ part of $K[E^*\otimes \mathcal{G}^2H^*(\mathcal{A}_{\omega})]$.

First consider the restriction of $\mu$ to $K[E^*\otimes E^*]$. Note that here $K[E^*\otimes E^*]$ means the kernel of full graded symmetrisation $E^*\otimes E^*\to \mathcal{G}^4H^*(\mathcal{A}_{\omega})$. For arbitrary $e,e'\in E^*$,
\begin{align*}
\gamma(e)\alpha^2(e')-(-1)^{|e||e'|}\gamma(e')\alpha^2(e) &= \gamma(e)d\gamma(e')-(-1)^{|e||e'|+(|e'|-1)|e|}d\gamma(e)\gamma(e') \\
&= (-1)^{|e|-1}d\big(\gamma(e)\gamma(e')\big).
\end{align*}
So $\mu$ vanishes on graded anti-symmetric tensors, and it factors through the projection $K[E^*\otimes E^*]\to\mathcal{B}^*(\mathcal{A_{\omega}})$. Moreover, the induced map $\mathcal{B}^*(\mathcal{A_{\omega}})\to H^{*-1}(\mathcal{A}_{\omega})$ is exactly the Bianchi-Massey tensor, which is 0 when acting on $\mathcal{B}^{n+1}(\mathcal{A_{\omega}})$ by assumption. Therefore, $\mu\equiv 0$ on $K[E^*\otimes E^*]$, in particular on its degree $n+1$ part. As $\eta(e)\alpha^2(e')$ is exact for any $e,e'\in E^*$, $\mu'=0$ on the degree $n+1$ part of $K[E^*\otimes E^*]$ for any choice of $\gamma'$.

Let $q:\mathcal{G}^2H^*(\mathcal{A}_{\omega})\to \mathcal{G}^2(\text{coker}\,\omega)$ be the morphism induced by the projection $H^*(\mathcal{A}_{\omega})\to \text{coker}\,\omega$ with kernel $\theta\ker\omega$, and $L^*$ be the preimage of $\theta\ker\omega$ under the multiplication map $c:\mathcal{G}^2H^*(\mathcal{A}_{\omega})\to H^*(\mathcal{A}_{\omega})$. Then $c\circ q(L^*)=0$ because its image is the projection of $c(L^*)$. Hence, $q(L^*)\subset E^*$. So there exists an induced map
$$
Q: E^*\otimes L^*\to E^*\otimes E^*,\quad e\otimes a\mapsto qe\otimes qa.
$$
Moreover, $Q$ sends $K[E^*\otimes L^*]$ into $K[E^*\otimes E^*]$. Here $K[E^*\otimes L^*]$ also means the kernel of full graded symmetrisation $E^*\otimes L^*\to \mathcal{G}^4H^*(\mathcal{A}_{\omega})$.

Let $e\in E^*$ and $a\in L^*$ with $|e|+|a|=n+1$. Then every term of $a-qa$ has a factor in $\theta\ker\omega$. Thus, $\alpha^2(a-qa)\in\mathbf{I}(\theta)$. As $\gamma(e)$ is also in $\mathbf{I}(\theta)$, we have $\gamma(e)\alpha^2(a)=\gamma(e)\alpha^2(qa)$. For the same reason,
\begin{align*}
\gamma(e)\alpha^2(qa) &= (-1)^{|e|}d\big(\gamma(e)\gamma(qa)\big)-(-1)^{|e|}\alpha^2(e)\gamma(qa) \\
&= (-1)^{|e|}d\big(\gamma(e)\gamma(qa)\big)-(-1)^{|e|}\alpha^2(qe)\gamma(qa) \\
&= (-1)^{|e|}d\big(\gamma(e)\gamma(qa)\big)-(-1)^{|e|}d\big(\gamma(qe)\gamma(qa)\big)+\gamma(qe)\alpha^2(qa).
\end{align*}
Therefore, $\mu(e\otimes a)=[\gamma(e)\alpha^2(a)]=[\gamma(qe)\alpha^2(qa)]=\mu\circ Q(e\otimes a)$. As discussed above, $\mu$ vanishes on $K[E^*\otimes E^*]$. So $\mu(K[E^*\otimes L^*])=\mu\circ Q(K[E^*\otimes L^*])\subset\mu(K[E^*\otimes E^*])=0$.

Now let $p=id\otimes c: K[E^*\otimes\mathcal{G}^2H^*(\mathcal{A}_{\omega})]\to E^*\otimes H^*(\mathcal{A}_{\omega})$ sending $e\otimes(x\cdot y)$ to $e\otimes xy$. Then $\ker p$ is exactly $K[E^*\otimes E^*]$. Hence, $p$ induces a morphism $\bar{\mu}:\im p\to H^*(\mathcal{A}_{\omega})$ of degree $-1$. As $\mu$ vanishes on $K[E^*\otimes L^*]$, $\bar{\mu}=0$ on $\im p\cap (E\otimes\theta\ker\omega)$. So we can extend $\bar{\mu}$ to all $E^*\otimes H^*(\mathcal{A}_{\omega})$ such that it vanishes on $E^*\otimes\theta\ker\omega$.

For each $e\in E^i$, $\bar{\mu}$ induces a morphism
$$
H^{n+1-i}(\mathcal{A_{\omega}})\to \mathbb{K}, \quad x\mapsto -\alpha_H\circ\bar{\mu}(e\otimes x)
$$
where $\alpha_H$ is the Poincar\'e class of $H^*(\mathcal{A_{\omega}})$. Let $\delta(e)$ be the unique element in $H^{i-1}(\mathcal{A_{\omega}})$ such that this morphism equals to $\alpha_H\frown\delta(e)$. We claim that $\delta(e)\in\theta\ker\omega$.

Write $\delta(e)=x_0+\theta y_0$ for some $x_0\in\text{coker}\,\omega$ and $y_0\in\ker\omega$. If $x_0\neq 0$, there exist some $x_1\in\text{coker}\,\omega$ and $y_1\in\ker\omega$ such that $(\alpha_H\frown x_0)(x_1+\theta y_1)\neq 0$, i.e. $x_0(x_1+\theta y_1)$ is non-trivial in $H^n(\mathcal{A_{\omega}})$. By the discussion at the beginning of the proof, $H^n(\mathcal{A_{\omega}})\subset \theta\ker\omega$. So $x_0x_1$ has to be 0. It follows that $\delta(e)\theta y_1=(x_0+\theta y_0)\theta y_1=x_0\theta y_1\neq 0$. On the other hand, $\delta(e)\theta y_1=-\alpha_H\circ\bar{\mu}(e\otimes\theta y_1)=0$ as $\bar{\mu}=0$ on $E^*\otimes\theta\ker\omega$, which is a contradiction. Hence, $x_0$ must be 0 and $\delta(e)\in\theta\ker\omega$.

Let $\eta=\alpha\circ\delta:E^*\to \mathcal{Z}^{*-1}(\mathcal{A_{\omega}})$. It represents $\delta$ and its image is in $\mathbf{I}(\theta)$. Then $\gamma'\alpha^2=(\gamma+\eta)\alpha^2$ induces a map $\mu':K[E^*\otimes \mathcal{G}^2H^*(\mathcal{A}_{\omega})]\to H^{*-1}(\mathcal{A}_{\omega})$. For arbitrary $\sum e_j\otimes a_j\in K[E^*\otimes \mathcal{G}^2H^*(\mathcal{A}_{\omega})]$,
\begin{align*}
\mu'\left(\sum e_j\otimes a_j\right) &= \mu\left(\sum e_j\otimes a_j\right)+\sum [\eta(e_j)\alpha^2(a_j)] \\
&= \bar{\mu}\left(\sum e_j\otimes c(a_j)\right)+\sum\delta(e_j)c(a_j) = 0.
\end{align*}

Finally, we prove that the uniform Massey product $\mathcal{T}'$ defined by $\gamma'$ is identically 0. This follows from $K[E^*\otimes H^*(\mathcal{A}_{\omega})]\otimes H^*(\mathcal{A}_{\omega})\subset K[E^*\otimes H^*(\mathcal{A}_{\omega})\otimes H^*(\mathcal{A}_{\omega})]$ and
$$
\mathcal{T'}\left(\sum e_j\otimes x_j\right) y=\left[\sum\gamma(e_j)\alpha(x_j)\alpha(y)\right]=\mu'\left(\sum e_j\otimes (x_j\cdot y)\right)=0
$$
for $e_j\in E^*$, $x_j,y\in H^*(\mathcal{A}_{\omega})$ and $|e_j|+|x_j|+|y|=n+1$. Thus, $\alpha_H\frown \mathcal{T'}\left(\sum e_j\otimes x_j\right)=0$ and then $\mathcal{T'}\left(\sum e_j\otimes x_j\right)=0$. As $\im\gamma'\subset\mathbf{I}(\theta)$ by construction, the discussion in the first paragraph of the proof shows that $\gamma'\alpha$ has to be 0 on $K[E^*\otimes H^*(\mathcal{A}_{\omega})]$. 
\end{proof}

\begin{lem}\label{vanishing uniform massey to formal}
Let $\mathcal{A}$ be any connected CDGA. Suppose that there are choices of $\alpha:H^*(\mathcal{A})\to \mathcal{A}^*$ and $\gamma:E^*(\mathcal{A})\to \mathcal{A}^{*-1}$ such that the restriction of $\gamma\alpha$ to $K[E^*(\mathcal{A})\otimes H^*(\mathcal{A})]$ and $\gamma^2: E^*(\mathcal{A})\otimes E^*(\mathcal{A})\to \mathcal{A}^{*-2}, e\otimes e' \mapsto \gamma(e)\gamma(e')$ both vanish. Then $\mathcal{A}$ is formal.
\end{lem}
\begin{proof}
For simplicity, in this proof we will use $H^*$ and $E^*$ to denote $H^*(\mathcal{A})$ and $E^*(\mathcal{A})$ respectively.

Construct a Sullivan model $\phi:\Lambda V \to H^*$ as follows. Start from a $d$-closed space $V^*_0$ of lower index 0. Then construct $V^*_p$ inductively such that $dV^k_p\subset (\Lambda V_{<p}^*)^{k+1}_{p-1}$. Here $V_{<p}^*$ denotes $V_0^*\oplus V_1^*\oplus\ldots\oplus V_{p-1}^*$, and $(\Lambda V_{<p}^*)_{p-1}$ is the subspace of elements of lower index $p-1$, where the lower index of an element in $V_{i_1}^*\otimes\ldots\otimes V_{i_k}^*$ is $i_1+\ldots+i_k$.

Precisely, let $V_0^*=H^+$ as a graded vector space, i.e. for each $x\in H^+$, there exists a unique $v_x\in V_0^*$ such that $\phi(v_x)=x$. Then let $V_1^k$ be spanned by $\left\{ w_{(x\cdot y)} |  (x\cdot y)\in\big(\mathcal{G}^2H^+\big)^{k+1} \right\}$, and set $dw_{(x\cdot y)}=v_{xy}-v_xv_y$. Thus, $\phi$ induces an isomorphism from the quotient space $\Lambda V_0^*$ over the ideal generated by $dV_1^*$ to $H^*$.

For $p\geq 2$, define $V_p^*$ such that $d: V^k_p \to (\Lambda V_{<p}^*)^{k+1}_{p-1}\cap \ker d$ is an isomorphism. Let $V=\oplus_{p\geq 0} V_p^*$ and $\phi(V_p^*)=0$ for all $p\geq 1$. Then $\phi:\Lambda V \to H^*$ is a quasi-isomorphism.

Now we construct a CDGA morphism $f:\Lambda V\to \mathcal{A}^*$ as follows. For $v_x\in V_0^*$, let $f(v_x)=\alpha(x)$. For $w_{(x\cdot y)}\in V_1^*$, let $f(w_{(x\cdot y)})=\gamma\big( (1\cdot xy)-(x\cdot y) \big)$. Then $f$ is well-defined on $\Lambda(V_0^*\oplus V_1^*)$. We claim that $f=0$ when acting on the closed elements of lower index 1, so that we can make $f=0$ on $V_2^*$.

To prove the claim, we will use the fact that for $x,y,z\in H^+$,
$$
(-1)^{|x|}v_x w_{(y\cdot z)}-w_{(x\cdot y)} v_z-w_{(xy\cdot z)}+w_{(x\cdot yz)}
$$
is a closed element in $\ker f$. A straight forward calculation shows that it is closed. And $f$ acting on this element is $\gamma\alpha$ acting on
$$
\Phi = (-1)^{|x|(|y|+|z|)} (1\cdot yz-y\cdot z)\otimes x-(1\cdot xy-x\cdot y)\otimes z-(1\cdot xyz-xy\cdot z)\otimes 1+(1\cdot xyz-x\cdot yz)\otimes 1.
$$
$\Phi$ is in $E^*\otimes H^*$ and its full symmetrisation is
\begin{align*}
& \quad (-1)^{|x|(|y|+|z|)} (1\cdot yz\cdot x-y\cdot z\cdot x)-(1\cdot xy\cdot z-x\cdot y\cdot z) \\
& \quad -(1\cdot xyz\cdot 1-xy\cdot z\cdot 1)+(1\cdot xyz\cdot 1-x\cdot yz\cdot 1) \\
& =0.
\end{align*}
Therefore, $\Phi\in K[E^*\otimes H^*]$ and $\gamma\alpha(\Phi)=0$ by hypothesis.

Now let $\Psi$ be a closed element in $\Lambda(V_0^*\oplus V_1^*)$ of lower index 1. So $\Psi\in V_1^*\cdot\Lambda V_0^*$. Choose a basis $\{ x_i\}$ of $H^+$, then $\Psi$ can be written as a linear combination of $w_{(x_{i_1}\cdot x_{i_2})}v_{x_{i_3}}\ldots v_{x_{i_l}}, l\geq 2$. Suppose $k$ is the largest integer such that $\Psi$ has a term in $V_1\cdot\Lambda^k V_0^*$, and there are $m_k$ terms in $V_1^*\cdot\Lambda^k V_0^*$ under this writing.

When $k\geq 1$, without loss of generality we may assume that $\Psi$ has a term $w_{(x_1\cdot x_2)}v_{x_3}\ldots v_{x_{k+2}}$. As $\Psi$ is closed, there must be another term such that its differential contains $v_{x_1}\ldots v_{x_{k+2}}$. Moreover, this term is also in $V_1\cdot\Lambda^k V_0$ since $dV_1\cdot\Lambda^j V_0\subset \Lambda^{\leq j+2}V_0$ for all $j\geq 0$. So it is of the form $Cw_{(x_{i_1}\cdot x_{i_2})}v_{x_{i_3}}\ldots v_{x_{i_{k+2}}}$, where $\{i_1,\ldots,i_{k+2}\}$ is a reshuffle of $\{1,\ldots,k+2\}$ and $C\in\mathbb{K}$ is a constant. Then either one of $1,2$ is in $\{i_1,i_2\}$, or $1,2\notin\{i_1,i_2\}$.

\textbf{Case 1.} $\{1,2\}\cap\{i_1,i_2\}\neq\emptyset$.

We may assume that the other term is $Cw_{(x_2\cdot x_3)}v_{x_1}v_{x_4}\ldots v_{x_{k+2}}$. Set
$$
\Psi'=\Psi+\big((-1)^{|x_1|}v_{x_1} w_{({x_2}\cdot {x_3})}-w_{(x_1\cdot x_2)} v_{x_3}-w_{(x_1x_2\cdot x_3)}+w_{(x_1\cdot x_2x_3)}\big)v_{x_4}\ldots v_{x_{k+2}}.
$$
Then $d\Psi'=d\Psi=0$ and $f(\Psi')=f(\Psi)$. Moreover, $\Psi'$ has at most $m_k-1$ terms in $V_1\cdot\Lambda^k V_0$.

\textbf{Case 2.} $\{1,2\}\cap\{i_1,i_2\}=\emptyset$.

We may assume that the other term is $Cw_{(x_3\cdot x_4)}v_{x_1}v_{x_2}v_{x_5}\ldots v_{x_{k+2}}$. Set
\begin{align*}
\Psi' &= \Psi+\big((-1)^{|x_1|}v_{x_1} w_{({x_2}\cdot {x_3})}-w_{(x_1\cdot x_2)} v_{x_3}-w_{(x_1x_2\cdot x_3)}+w_{(x_1\cdot x_2x_3)}\big)v_{x_4}\ldots v_{x_{k+2}} \\
&\quad +(-1)^{|x_1|}v_{x_1}\big((-1)^{|x_2|}v_{x_2} w_{({x_3}\cdot {x_4})}-w_{(x_2\cdot x_3)} v_{x_4}-w_{(x_2x_3\cdot x_4)}+w_{(x_2\cdot x_3x_4)}\big)v_{x_5}\ldots v_{x_{k+2}}.
\end{align*}
Then we also have $d\Psi'=d\Psi=0$, $f(\Psi')=f(\Psi)$, and $\Psi'$ has at most $m_k-1$ terms in $V_1\cdot\Lambda^k V_0$.

Continue this process we can finally find some closed $\Psi''\in V_1^*$ such that $f(\Psi'')=f(\Psi)$. But by construction the only closed element in $V_1^*$ is 0. So $\Psi''=0$ and $f(\Psi)=0$. The claim is proved.

Therefore, we can set $f(V_2^*)=0$. On the other hand, $f(V_1^*) \subset \im\gamma$. By hypothesis  $f(\Lambda^2 V_1^*)\subset \im\gamma^2=0$. Hence, $f=0$ on $\Lambda(V_{\leq 2}^*)_2$ and we can set $f(V_3^*)=0$. Inductively we have $f=0$ on $\Lambda(V_{\leq p}^*)_p$ and set $f(V_{p+1}^*)=0$ for $p\geq 2$.

It remains to verify that $f$ is a quasi-isomorphism. For each $x\in H^+$, there exists some closed $v_x\in V_0^*$ such that the cohomology class $[f(v_x)]=[\alpha(x)]=x$. Together with $f(1)=1$ we have proved the surjectivity.

For injectivity, suppose that $\Omega\in\Lambda V$ is closed and $f(\Omega)$ is exact. Observe that $d$ decrease the lower index by 1 homogeneously and all the closed elements of positive lower index are exact. So the cohomology class $[\Omega]$ has a representative in $\Lambda V_0^*$. Moreover, by construction for any $x,y\in H^+$, the cohomology class $[v_{xy}]=[v_xv_y]$. It follows that $[\Omega]$ has a representative in either $V_0^*$ or $\mathbb{K}$. But $f$ acting on $V_0^*\oplus\mathbb{K}$ is non-exact except $f(0)$. So this representative is 0 and $\Omega$ is exact.

Therefore, $\phi$ and $f$ give a CDGA equivalence between $H^*$ and $\mathcal{A}^*$. 
\end{proof}

\begin{rmk}
We can also construct an $A_{\infty}$-quasi-isomorphism from $H^*$ to $\mathcal{A}^*$, by setting $f_1=\alpha$, $f_2(x,y)=\gamma\big((1\cdot xy-x\cdot y)\big)$, and $f_p=0$ for all $p\geq 3$.
\end{rmk}

The CDGA $\mathcal{A}_{\omega}$ in Lemma \ref{vanishing uniform massey}, satisfies the hypothesis of Lemma \ref{vanishing uniform massey to formal}. So $\mathcal{A}_{\omega}$ is formal and we have the following theorem.
\begin{thm}\label{B-M tensor to formality}
Let $\mathcal{A}$ be a formal CDGA, and $\mathcal{A}_{\omega}=\mathcal{A}\otimes \Lambda\theta$ be a Poincar\'e algebra where $\theta^2=0$ and $d\theta=\omega\in\mathcal{A}$. Then $\mathcal{A}_\omega$ is formal if and only if its Bianchi-Massey tensor vanishes.
\end{thm}

Let $\pi:X\to M$ be an orientable $S^k$-bundle. We have the following CDGA equivalence.
\begin{align*}
\Omega^*(X)\simeq
\begin{cases}
\Omega^*(M)\otimes\Lambda(\theta), \quad\, d\theta=e, & k\text{ is odd,}\\
\Omega^*(M)\otimes\Lambda(\theta,\theta'), \, d\theta=0, d\theta'=\theta^2+\frac{1}{4}p, & k\text{ is even.}
\end{cases}
\end{align*}
Here $|\theta|=k$, and $|\theta'|=2k-1$. $[e]\in H^{k+1}(M)$ is the Euler class, and $[p]\in H^{2k}(M)$ is the rational $2k$-th Pontryagin class of the sphere bundle.

When $M$ is simply connected, the proof can be found in \cite[Example 4, Page 202]{FHT}. For general manifolds, the case that $k$ is odd is proved in \cite[Appendix]{tt}, and a similar proof works when $k$ is even. Moreover, this equivalence still holds when $X$ is a spherical Serre fibration, as long as $\pi_1(M)$ acts nilpotently on $H^*(S^k)$ \cite[Theorem 20.3]{Halperin}, which is equivalent to the holonomy action of $\pi_1(M)$ on the fiber $S^k$ preserves the orientation.

Therefore, when $M$ is compact, orientable and formal, $\Omega^*(X)$ is equivalent to a CDGA $\mathcal{A}_{\omega}$ satisfying the hypothesis of Theorem \ref{B-M tensor to formality}. When $k$ is odd, $\mathcal{A}=H^*(M)$ and $\omega=[e]$. When $k$ is even, $\mathcal{A}=H^*(M)\otimes\Lambda(\theta)$, $\mathcal{A}_{\omega}=\mathcal{A}\otimes\Lambda(\theta')$, and $\omega=\theta^2+\frac{1}{4}[p]$. In both cases the formality of $X$ is determined by the Bianchi-Massey tensor of $\mathcal{A}_{\omega}$.

Moreover, if $k$ is even, the kernel of multiplying by $\omega=\theta^2+\frac{1}{4}[p]$ in $\mathcal{A}=H^*(M)\otimes\Lambda(\theta)$ is 0. By the discussion in the third paragraph of the proof of Lemma \ref{vanishing uniform massey}, the uniform Massey product $\mathcal{T}$ of $\mathcal{A}_{\omega}=H^*(M)\otimes\Lambda(\theta,\theta')$ is trivial as $\ker\omega=0$. Then it is formal by Lemma \ref{vanishing uniform massey to formal}. Thus, we have the following statement.

\begin{thm}
Suppose $M$ is a compact orientable formal manifold, and $\pi:X\to M$ is an orientable $S^k$-bundle. Then $X$ is formal if and only if the Bianchi-Massey tensor of $\Omega^*(X)$ vanishes. Moreover, when $k$ is even, $X$ is always formal.
\end{thm}

\subsection{A special case: Euler class is of the top degree}

In this subsection, we will consider the case that the Euler class is a top degree cohomology class of a formal manifold. An example of such bundle is the unit tangent bundle. Let $M$ be a compact orientable formal manifold. Equip the tangent bundle $TM$ with a metric, then the vectors of norm 1 form a sphere bundle $UTM$, whose Euler class $\chi(M)[\omega]$. Here $\chi(M)$ is the Euler Characteristic and $\omega$ is a volume form. $UTM$ is called the unit tangent bundle of $M$.

We will explore when such bundles are formal. When the base is of odd dimension, the fiber is an even dimensional sphere. By the discussion of the previous subsection, we already know that the bundle is formal. So the non-trivial case only happens on even dimensional base manifolds.

\begin{lem}
Suppose $\mathcal{A}$ is a $2n$-dimensional Poincar\'e CDGA with trivial differential, and $\omega\in\mathcal{A}^{2n}$ be non-zero. We also assume that $\mathcal{A}^i=0$ if $i<0$, and $\mathcal{A}^0=\mathbb{K}$. If $\mathcal{A}_{\omega}$ is formal, then the multiplication map $\mathcal{A}^i\otimes\mathcal{A}^j\to\mathcal{A}^{i+j}$ is injective for all $i,j\leq n$.
\end{lem}
\begin{proof}
If $i$ or $j=0$, the multiplication map is an isomorphism. Assume that the multiplication map has a non-trivial kernel in $\sum_{r=1}^k x_r\otimes y_r\in\mathcal{A}^i\otimes\mathcal{A}^j$, where $0<i,j\leq n$. $\{ x_r \}$ can be chosen linearly independent with $k\geq 1$. A similar procedure can make $\{ y_r \}$ linearly independent: If $y_r=c_1y_1+\ldots +c_{r-1}y_{r-1}$, $x_1\otimes y_1+\ldots+x_r\otimes y_r$ can be rewritten as $(x_1+c_1x_r)\otimes y_1+\ldots+(x_{r-1}+c_{r-1}x_r)\otimes y_{r-1}$. $\{ x_1+c_1x_r,\ldots,x_{r-1}+c_{r-1}x_r \}$ is also linearly independent.

Take $x_1^*\in\mathcal{A}^{2n-i},y_1^*\in\mathcal{A}^{2n-j}$ such that $x_rx_1^*=y_ry_1^*=\delta_{1r}\omega$. Then $n\leq x_1^*,y_1^*<2n$. Hence, $x_1^*,y_1^*$ are non-exact in $\mathcal{A}_{\omega}$, but the products $x_rx_1^*,y_ry_1^*$ are either 0 or $\omega$, which are both exact in $\mathcal{A}_{\omega}$. So we can set $\alpha:H^*(\mathcal{A}_{\omega})\to\mathcal{A}_{\omega}$ such that
$$
\alpha([x_r])=x_r, \quad \alpha([y_r])=y_r, \quad \alpha([x_1^*])=x_1^*, \quad \alpha([y_1^*])=y_1^*,
$$
and $\gamma:E^*(\mathcal{A}_{\omega})\to\mathcal{A}_{\omega}$ such that
$$
\gamma\left(\sum_{r=1}^k[x_r]\cdot[y_r]\right)=0, \quad \gamma([x_r]\cdot[x_1^*])=\gamma([y_r]\cdot[y_1^*])=\delta_{1r}\theta.
$$
Then $\left( (\sum_{r=1}^k[x_r]\cdot[y_r])\cdot ([x_1^*]\cdot[y_1^*]) \right)-(-1)^{j(2n-i)}\sum_{r=1}^k \Big( ([x_r]\cdot[x_1^*])\cdot ([y_r]\cdot[y_1^*]) \Big)$ is in $K[\mathcal{G}^2\mathcal{G}^2 H^*(\mathcal{A}_{\omega})]$, and
\begin{align*}
& \quad \mathcal{F}\left( \left( (\sum_{r=1}^k[x_r]\cdot[y_r])\cdot ([x_1^*]\cdot[y_1^*]) \right)-(-1)^{j(2n-i)}\sum_{r=1}^k \Big( ([x_r]\cdot[x_1^*])\cdot ([y_r]\cdot[y_1^*]) \Big) \right) \\
&= -(-1)^{ij}[\gamma ([x_1]\cdot[x_1^*])y_1y_1^*]-(-1)^{ij}(-1)^{2n\cdot 2n}[\gamma ([y_1]\cdot[y_1^*])x_1x_1^*] \\
&= -(-1)^{ij}2[\theta\omega].
\end{align*}
Therefore, $\mathcal{A}_{\omega}$ is non-formal.
\end{proof}

\begin{lem}
Suppose $H^*$ is a $2n$-dimensional Poincar\'e graded algebra, $H^0=\mathbb{K}$, and $H^i$ is nontrivial only when $0\leq i\leq 2n$. If the multiplication map $H^i\otimes H^j\to H^{i+j}$ is injective for all $i,j\leq n$, then $H^*=\mathbb{K}[x]/(x^p)$ is a quotient of the polynomial ring with a single variable.
\end{lem}
\begin{proof}
First observe that $\dim H^{2i}\leq 1$ and $\dim H^{2i+1}=0$. If $x,y\in H^{2i}$ are linearly independent for some $0<2i\leq n$, then $x\otimes y-y\otimes x$ will be a non-trivial element in the kernel of the multiplication map $H^{2i}\otimes H^{2i}\to H^{4i}$. If $z\in H^{2i+1}$ is non-zero for some $0<2i+1\leq n$, then $z\otimes z$ will be a non-trivial element in the kernel of the multiplication map $H^{2i+1}\otimes H^{2i+1}\to H^{4i+2}$. As $H^*$ is $2n$-dimensional Poincar\'e, for $n<2i\leq 2n$ we have $\dim H^{2i}=\dim H^{2n-2i}\leq 1$, and for $n<2i+1\leq 2n$ we have $\dim H^{2i+1}=\dim H^{2n-2i-1}=0$. 

Let $S=\{ i\in\mathbb{Z} \,| \dim H^i=1 \}$, $k$ be the smallest positive integer in $S$, and $T$ be the image of $S$ under the projection $\mathbb{Z}\to\mathbb{Z}/k\mathbb{Z}$. For $i,j\in S$ with $i,j\leq n$, the hypothesis that the multiplication $H^i\otimes H^j\to H^{i+j}$ is injective implies $i+j\in S$. In particular, $ik\in S$ for $0\leq i\leq 2[\frac{n}{k}]$. Let $x$ be a generator of the vector space $H^k$, then $x^i$ generates $H^{ik}$ for such $i$.

We claim that every $a\in T$ has a representative $m\leq n$ in $S$. Indeed, when $m>n$ is a representative of $a$ with $m\in S$, we have $2n-m<n$ and $2n-m\in S$ as $H^*$ is Poincar\'e. On the other hand, $[\frac{n}{k}]k\leq n$ is also in $S$ as discussed in last paragraph, so $(2n-m)+[\frac{n}{k}]k \in S$. By the definition of $k$, we have $2n-m\geq k$. Then $(2n-m)+[\frac{n}{k}]k \geq (1+[\frac{n}{k}])k>n$. Thus, we can set $m'=2n-\big( (2n-m)+[\frac{n}{k}]k \big)=m-[\frac{n}{k}]k$, which is also a representative of $a$. Moreover, $m'\in S$ and $m'<n$.

For arbitrary $a_1,a_2\in T$, there are respective representatives $m_1,m_2\leq n$ in $S$. Then $m_1+m_2\in S$ and $a_1+a_2\in T$. Since the finite subset $T$ of $\mathbb{Z}/k\mathbb{Z}$ is closed under addition, it is a subgroup. So $T=\mathbb{Z}/c\mathbb{Z}$ for some positive integer $c$, and $\frac{k}{c}$ is the number of elements in $T$.

Choose the largest $m\leq n$ in $S$. We claim that $n-m<c$. Otherwise, the class of $m+c$ in $\mathbb{Z}/k\mathbb{Z}$ is also in $T$, and it has a representative $m'\leq n$ in $S$. By assumption, $m'\leq m<m+c$. So $m+c-m'$ is a positive integer divisible by $k$. Moreover, $m+c-m'\leq m+c<n< (1+[\frac{n}{k}])k\leq 2[\frac{n}{k}]k$ as $[\frac{n}{k}]\geq 1$ clearly. Hence, $m+c-m'\in S$. It follows that $m+c=m'+(m+c-m')\in S$, which is a contradiction.

If $T\neq\{0\}$, the number of its elements $\frac{k}{c}\geq 2$. Together with $n-m<c$ we have $2n-2m<k$. On the other hand, $2m\in S$ as $m<n$ is in $S$, then $2n-2m\in S$, which contradicts to the hypothesis of $k$. So $T$ has to be $\{0\}$ and $k$ divides $2n$.

When $k$ divides $n$, $2[\frac{n}{k}]k=2n$. The discussion in the second paragraph shows that $H^*=\langle x^i| 0\leq|x^i|\leq 2n\rangle$. When $k$ does not divide $n$, $(2[\frac{n}{k}]+1)k=2n$. In the special case that $k=2n$, $H^*=H^0\oplus H^{2n}$. So any generator of $x$ the space $H^{2n}$ makes $H^*=\mathbb{K}[x]/(x^2)$. In the general case, $H^k,H^{2k},\ldots,H^{2n-k}$ are all generated by some $x^i$, and it remains to verify that the generator of $H^{2n}$ is also a power of $x$. Since $H^*$ is Poincar\'e, this generator is the product of $x$ and a non-trivial element in $H^{2n-k}$. But the only generator of $H^{2n-k}$ is $x^i$ with $i=2[\frac{n}{k}]$, so $x^{i+1}$ is a generator of $H^{2n}$. This completes the proof.
\end{proof}

Combining the lemmas above, we have the following theorem.

\begin{thm}\label{formality of extended by top degree}
Suppose that $\mathcal{A}$ is a $2n$-dimensional Poincar\'e CDGA, which is formal and connected. Let $\omega\in\mathcal{A}^{2n}$ such that $[\omega]$ is non-zero. Set $\mathcal{A}_{\omega}=\mathcal{A}\otimes\Lambda\theta$ with $d\theta=\omega$. If $\mathcal{A}_{\omega}$ is formal, then $H^*(\mathcal{A})=\mathbb{K}[x]/(x^p)$ is a quotient of the polynomial ring with a single variable.
\end{thm}

\begin{cor}
Let $M$ be a compact orientable formal manifold. Its unit tangent bundle $UTM$ is formal if and only if one of the following statement holds

1. The Euler characteristic $\chi(M)=0$.

2. $H^*(M)=\mathbb{K}[x]/(x^p)$ is a quotient of the polynomial ring with a single variable.
\end{cor}
\begin{proof}
If $\chi(M)\neq 0$, then $M$ is an even dimensional manifold and the Euler class of $UTM$ is non-trivial. Theorem \ref{formality of extended by top degree} states that $H^*(M)$ has to be a quotient of the polynomial ring with a single variable.

Conversely, when $\chi(M)=0$, the Euler class of $UTM$ is trivial. It follows that $UTM$ is a trivial bundle and formal. When $H^*(M)=\mathbb{K}[x]/(x^p)$, $\Omega^*(UTM)$ is equivalent to $\left(\mathbb{K}[x]/(x^p)\right)\otimes\Lambda\theta$ with $d\theta=x^{p-1}$. $H^*(UTM)$ is spanned by $1,x,\ldots,x^{p-2},\theta x,\theta x^2,\ldots,\theta x^{p-1}$. Let $\mathcal{M}=\Lambda(u,v,w)$ such that $|u|=|x|$, $|v|=|\theta|$, $|w|=|\theta x|$, $du=dw=0$ and $dv=u^{p-1}$. Then we can define quasi-isomorphisms $\phi:\mathcal{M}\to H^*(UTM)$ and $f:\mathcal{M}\to \left(\mathbb{K}[x]/(x^p)\right)\otimes\Lambda(\theta)$ such that
$$
\phi(u)=f(u)=x,\quad \phi(v)=0,\quad f(v)=\theta,\quad \phi(w)=f(w)=\theta x.
$$
This gives the CDGA equivalence between $H^*(UTM)$ and $\Omega^*(UTM)$.
\end{proof}

\begin{cor}
Let $M$ be a compact orientable formal manifold. If the Euler characteristic $\chi(M)<0$, then $UTM$ is non-formal.
\end{cor}

\begin{ex}
As a simple example, the circle bundles over Riemann surfaces distinguish the different cases above.
\begin{center}

\begin{tabular}{|c|c|c|}
\hline
genus & unit tangent bundle & Euler class is volume form \\
\hline
0 & formal & formal \\
\hline
1 & formal & non-formal \\
\hline
$\geq 2$ & non-formal & non-formal \\
\hline
\end{tabular}
\end{center}
\end{ex}

\section{An Obstruction to the Formality of General Sphere Bundles}

If the base manifold $M$ is compact and formal, we have established when the sphere bundle $X$ is formal. One may consider the case for non-formal $M$. We will discuss it in this section, and give an obstruction to the formality of $X$.

Let $\mathcal{M}$ be the minimal Sullivan model of a CDGA $\mathcal{A}$. For a closed element $\omega\in\mathcal{A}$, there exists a closed element in $\mathcal{M}$ whose cohomology class is corresponding to $[\omega]$. We will also write this element in $\mathcal{M}$ as $\omega$. Regardless of the choice of $\omega$, we have the quasi-isomorphism
$$
\mathcal{M}_{\omega}= \mathcal{M}\otimes\Lambda \theta \simeq \mathcal{A}\otimes\Lambda \theta = \mathcal{A}_{\omega},
$$
where $d\theta=\omega$. In this section we assume that $|\omega|$ is even. So $\theta^2$ is automatically 0.

Whether the Sullivan model $\mathcal{M}_{\omega}$ is minimal depends on whether $\omega\in\mathcal{M}$ is reducible. Note that the representatives of $[\omega]\in H^*(\mathcal{M})$ are either all reducible or all irreducible, since exact elements in $\mathcal{M}$ are reducible by the definition of minimal Sullivan algebra in Section 2.2.

\begin{prop}
When $\omega\in\mathcal{M}$ is reducible, $\mathcal{M}_{\omega}$ is a minimal Sullivan model of $\mathcal{A}_{\omega}$.
\end{prop}
\begin{proof}
Suppose $\mathcal{M}=\Lambda V^*$ and $V^*=\langle v_{\alpha}\rangle$. Write $\|v_{\alpha}\|=\alpha$ as the index of $v_{\alpha}$. Then set $\|\theta\|>\alpha$ for all $|v_{\alpha}|\leq|\theta|$, and $\|\theta\|<\alpha$ for all $|v_{\alpha}|<|\theta|$. So $\{v_{\alpha}\}\cup\{\theta\}$ is a well-ordered set. It is straightforward to verify that $\mathcal{M}_{\omega}=\Lambda(V^*\oplus \langle\theta\rangle)$ is a minimal Sullivan algebra.
\end{proof}

\begin{thm}\label{HL and reducible non-formal}
Suppose $\omega$ is a closed element in a minimal Sullivan algebra $\mathcal{M}$ satisfying the following conditions.

1. $|\omega|=2r$ for some odd integer $r$, and $[\omega]$ has a representative that can be written as
$$
\sum_{i=1}^k x_iy_i,
$$
where $x_i,y_i$ are all closed in $\mathcal{M}^r$.

2. There exists some $s\geq0$ such that $H^s(\mathcal{M})$ is non-trivial. Moreover, the morphism $\omega:H^s(\mathcal{M})\to H^{s+2r}(\mathcal{M})$ multiplying by $[\omega]$ is an isomorphism, and $\omega:H^{s-r}(\mathcal{M})\to H^{s+r}(\mathcal{M})$ is injective.

Then $\mathcal{M}_{\omega}=\mathcal{M}\otimes \Lambda\theta$ is non-formal, where $d\theta=\omega$.
\end{thm}
\begin{proof}
Assume that $\mathcal{M}_{\omega}$ is formal. By Theorem \ref{C oplus N decomposition} we can write $\mathcal{M}_{\omega}=\Lambda V^*$ and $V^*=C^*\oplus N^*$ such that all closed elements in $\mathbf{I}(N^*)$ are exact. 

Since $\omega:H^s(\mathcal{M})\to H^{s+2r}(\mathcal{M})$ is an isomorphism, $[\omega]\in H^{2r}(\mathcal{M})$ cannot be a trivial class. So $\mathcal{M}_{\omega}$ satisfies the hypothesis of Lemma \ref{choosing representatives of extension}, which allows us to reset $\omega$ and $\theta$. Then without loss of generality, we can assume that $\omega$ itself has the form $\sum_{i=1}^k x_iy_i$ in Condition 1, and $\theta\in \mathbf{I}(N^*)$.

We will first prove the following claim that will be used in proof.

\noindent\textit{Claim.} Suppose $\mathcal{M}$ satisfies the hypothesis of the theorem, and $\mathcal{M}_{\omega}$ is formal. Let $\beta\in\mathcal{M}^{s+r}$ be a closed element. If $\omega\beta=\sum w_jz_j$ where $w_j\in\mathcal{M}^{s+2r}$ are all exact in $\mathcal{M}$ and $z_j\in\mathcal{M}^r$ are all closed, then $\beta$ is exact in $\mathcal{M}$.

Since $w_j$ is exact in $\mathcal{M}$, it is also exact in $\mathcal{M}_{\omega}$ and can be written as $w_j=d(\xi_j+\theta\eta_j)$, where $\xi_j\in\mathcal{M}^{s+2r-1}$, $\eta_j\in\mathcal{M}^r$ and $\xi_j+\theta\eta_j\in\mathbf{I}(N^*)$. As $w_j=d\xi_j+\omega\eta_j+\theta\eta_j$ is in $\mathcal{M}$, we have $d\eta_j=0$ and $\omega\eta_j=w_j-d\xi_j$. It follows that $[\omega][\eta_j]=0$ in $H^{s+2r}(\mathcal{M})$. By the hypothesis of Condition 2, $[\eta_j]$ has to be 0.

Now we consider $\theta\beta-\sum (\xi_j+\theta\eta_j)z_j$. It is in $\mathbf{I}(N^*)$ as we assume $\theta$ and all $\xi_j+\theta\eta_j$ are in this ideal. Moreover,
$$
d\left(\theta\beta-\sum (\xi_j+\theta\eta_j)z_j\right) = \omega\beta-\sum w_jz_j =0.
$$
Hence, the closed element $\theta\beta-\sum (\xi_j+\theta\eta_j)z_j$ in $\mathbf{I}(N^*)$ has to be exact in $\mathcal{M}_{\omega}$, and we can write it as $d(\zeta+\theta\lambda)$ with $\zeta,\lambda\in\mathcal{M}$. Comparing the coefficient of $\theta$, we have $\beta-\sum \eta_jz_j=-d\lambda$. As we have shown that $z_j$ are all exact in $\mathcal{M}$, so is $\beta$. This proves the claim.

Next we prove by induction that for any $i$, there exists a closed but non-exact $a_i\in\mathcal{M}^s$ such that $[a_ix_j]=0$ in $H^{r+s}(\mathcal{M})$ for all $j\leq i$. The $i=0$ case is immediate from Condition 2, as $H^s(\mathcal{M})\neq 0$ and any representative of a non-trivial cohomology class can be taken as $a_0$.

Suppose we have found $a_0,\ldots,a_{i-1}$ for $i>0$. If $[a_{i-1}x_i]=0$ in $H^{r+s}(\mathcal{M})$ we can simply set $a_i=a_{i-1}$. Otherwise, since $\omega:H^s(\mathcal{M})\to H^{s+2r}(\mathcal{M})$ is an isomorphism according to Condition 2, for all $j>i$ we can write $[a_{i-1}x_ix_j]=[\omega b_j]$ with $b_j\in\mathcal{M}^s$ closed. Let $\beta=a_{i-1}x_i-\sum_{j>i}b_j y_j$. Then
$$
\omega\beta = \sum_{j\geq 1} a_{i-1}x_ix_jy_j-\sum_{j>i}\omega b_j y_j = \sum_{1\leq j\leq i-1} (a_{i-1}x_ix_j)y_j+\sum_{j>i} (a_{i-1}x_ix_j-\omega b_j)y_j.
$$
The last equation follows from $x_i^2=0$ as the degree of $x_i$ is odd. Since $a_{i-1}x_j$ are all exact in $\mathcal{M}$, so are the degree $s+2r$ elements $a_{i-1}x_ix_j$. As $(a_{i-1}x_ix_j-\omega b_j)y_j$ are also exact in $\mathcal{M}^{s+2r}$ and all $y_j$ are closed, we can apply the above claim to $\beta$. So $\beta$ is exact and $[a_{i-1}x_i]=\sum_{j>i}[b_j y_j]$ in $H^*(\mathcal{M})$.

$[a_{i-1}x_i]$ is assumed to be non-zero in $H^*(\mathcal{M})$, so there is some $j>i$ such that $[b_j]\neq 0$. Write $\omega b_j=a_{i-1}x_ix_j+d\eta$ for some $\eta\in\mathcal{M}$. Then for any $l\leq i$,
$$
\omega(b_jx_l) = a_{i-1}x_ix_jx_l+(d\eta)x_l=(a_{i-1}x_lx_i)x_j+(d\eta)x_l.
$$
When $l<i$, by the hypothesis of induction $a_{i-1}x_l$ is exact in $\mathcal{M}$, and hence so is $a_{i-1}x_lx_i$. When $l=i$, $a_{i-1}x_lx_i=0$. In either case $a_{i-1}x_lx_i$ and $d\eta$ are exact elements in $\mathcal{M}$ of degree $s+2r$, and $x_j,x_l\in\mathcal{M}^s$ are closed. Hence, we can apply the above claim to $b_jx_l$ to deduce that they are all exact in $\mathcal{M}$. Therefore, we can set $a_i=b_j$.

By induction we can find some non-zero $[a_k]\in H^s(\mathcal{M})$ such that $[a_kx_j]=0$ in $H^{r+s}(\mathcal{M})$ for all $j$. It follows that $[a_k\omega]=0$ in $H^{r+2s}(\mathcal{M})$. But by Condition 2, $\omega: H^r(\mathcal{M}) \to H^{r+2s}(\mathcal{M})$ is an isomorphism, which is a contradiction. So $\mathcal{M}_{\omega}$ cannot be formal.
\end{proof}

\begin{rmk}
The condition that $|\omega|\equiv 2 \,(\mathrm{mod}\,4)$ is necessary. For $|\omega|=4$, let $\mathcal{M}=\Lambda\langle x,\xi \rangle$ be the minimal Sullivan model of $\mathbb{C}P^2$, where $|x|=2$, $|\xi|=5$, $dx=0$, $d\xi=x^3$. Then $\omega=x^2$ induces an isomorphism $H^0(\mathcal{M})\to H^4(\mathcal{M})$. However, $\mathcal{M}_{\omega}$ is formal because we can set $C^*=\langle x,\theta x-\xi \rangle$ and $N^*=\langle \theta \rangle$.
\end{rmk}

On a symplectic manifold $(M,\omega)$, if $[\omega]$ is an integral cohomology class, there exists a circle bundle whose Euler class is $[\omega]$. This circle bundle is called the Boothby-Wang fibration.

\begin{cor}
Let $(M,\omega)$ be a connected symplectic manifold satisfying the hard Lefschetz property. Suppose $[\omega]$ is an integral and reducible cohomology class, i.e. there exists some $x_i,y_i\in H^1(M)$ such that $[\omega]=\sum x_i y_i$, then the Boothby-Wang fibration of $M$ is non-formal.
\end{cor}
\begin{proof}
Let $\mathcal{M}$ be the minimal Sullivan model of $M$, then there are $\omega,x_i,y_i\in\mathcal{M}$ whose cohomology classes are same as the corresponding elements in $H^*(M)$, and they satisfy $\omega=\sum x_i y_i$.

Suppose $\dim M=2n$, then $\omega:H^{n-1}(\mathcal{M})\to H^{n+1}(\mathcal{M})$ is isomorphic and $\omega:H^{n-2}(\mathcal{M})\to H^n(\mathcal{M})$ by the hard Lefschetz property. As $[\omega]^n\neq 0$, there exists some $[x_{i_1}y_{i_1}\ldots x_{i_n}y_{i_n}]\neq 0$. Hence $[x_{i_1}\ldots x_{i_{n-1}}]\neq 0.$ Then we can apply Theorem \ref{HL and reducible non-formal} to prove that $\mathcal{M}_{\omega}$ and the Boothby-Wang fibration of $M$ are non-formal.
\end{proof}

When the base manifold $M$ is formal, the condition that $[\omega]$ is reducible in $H^*(M)$ is equivalent to having a reducible representative $\omega_0$ in $\mathcal{M}$, the minimal Sullivan model of $M$. The reason is that $\mathcal{M}$ can be generated by some $C^*\oplus N^*$ and closed elements in $\mathbf{I}(N^*)$ are all exact. So the $\Lambda C^*$ part of $\omega_0$ is also a representative of $[\omega]$. However, the sufficiency of this weakened condition for general base manifolds remains unknown.

Besides, it is uncertain whether the above corollary for symplectic manifolds still holds without the hard Lefschetz property.

Finally, when $\omega$ is irreducible, the minimal Sullivan model of $\mathcal{M}_{\omega}$ is slightly different. In this case, $\omega$ can be chosen a generator of $\mathcal{M}$, i.e. $\mathcal{M}=\Lambda V^*, V^*=\langle v_{\alpha} \rangle$ and $\omega=v_\alpha$ for some $\alpha$. Let $V^*/\omega$ be a subspace spanned by all $v_{\alpha}$ except $\omega$, where the order of $v_{\alpha}$ is preserved. Then let $\mathcal{M}/\omega=\Lambda(V^*/\omega)$ and $\Pi:\mathcal{M}\to\mathcal{M}/\omega$ be the natural projection. $\Pi\circ d$ can be taken as the differential of $\mathcal{M}/\omega$.

\begin{prop}
$\mathcal{M}/\omega$ is a minimal Sullivan model of $\mathcal{M}_{\omega}$. The inclusion is a quasi-isomorphism.
\end{prop} 

Thus, the way of proving formality or finding obstructions for irreducible Euler classes may be quite different than the reducible case. Algebraically,  Amann and Kapovitch constructed a formal $S^3$-fibration with a non-formal base \cite[Page 21]{AK}. Although this example is infinite dimensional, it may provide a potential insight into the construction of a formal sphere bundle over a non-formal manifold. Additionally, it is interesting to investigate the condition under which there exists a formal sphere bundle over a non-formal manifold.

\bibliographystyle{plain} 
\bibliography{refs}

\vskip 1 cm
\noindent
{Beijing Institute of Mathematical Sciences and Applications, Huairou District, Beijing, China 101408}\\
{\it Email address:}~{\tt jiaweizhou@bimsa.cn}
\end{document}